\documentclass[english, 11pt]{amsart}
\usepackage{amssymb}
\usepackage{amsfonts}
\usepackage{amscd}
\usepackage[T1]{fontenc}
\usepackage{enumerate}%
\usepackage{color}
\usepackage[utf8]{inputenc}

\newcommand\blue[1]{{\color{blue}#1}}

\definecolor{immi}{rgb}{0,.6,.1}

\long\def\change#1{{#1}}

\newbox\removebox
\newcommand\remove[2]{%
\setbox\removebox=\ifmmode\hbox{$#2$}\else\hbox{#2}\fi%
\leavevmode
\rlap{\textcolor{#1}{\vrule height0.8ex depth-0.6ex width\wd\removebox}}%
\box\removebox
}
\long\def\bigremove#1{%
\par\setbox\removebox=\vbox{#1}%
\vbox{%
\vbox to0pt{\hbox{\tikz\draw[color=blue,thick] (0,0) -- (\wd\removebox,-\ht\removebox)  (\wd\removebox,0) -- (0,-\ht\removebox);}}
\box\removebox
}
}

\usepackage{mathrsfs} 

\makeatletter
\def\RFss@@#1{\RF^*_{\!*#1}}
\def\RFss@_#1{\RFss@@{,#1}}
\def\RFss{\@ifnextchar_{\RFss@}{\RFss@@{}}}
\makeatother

\def\GL{\mathrm{GL}}
\newcommand{\RF}{{\rm RF}}

\def\Supp{{\operatorname{Supp}}}

\def\lct{\operatorname{lct}}

\def\Tr{\operatorname{Tr}}
\def\Div{\operatorname{Div}}

\def\Sing{\operatorname{Sing}}
\def\Crit{\operatorname{Crit}}

\def\ac{{\overline{\rm ac}}}

\def\GL{\operatorname{GL}}

\def\Conv{\operatorname{Conv}}

\def\11{{\mathbf 1}}
\def\AA{{\mathbb A}}

\def\CC{{\mathbb C}}

\def\FF{{\mathbb F}}

\def\NN{{\mathbb N}}

\def\QQ{{\mathbb Q}}
\def\RR{{\mathbb R}}

\def\TT{{\mathbb T}}

\def\ZZ{{\mathbb Z}}

\def\cB{{\mathcal B}}

\def\cM{{\mathcal M}}
\def\cN{{\mathcal N}}
\def\cO{{\mathcal O}}

\def\cZ{{\mathcal Z}}
\def\llp{\mathopen{(\!(}}
\def\llb{\mathopen{[\![}}
\def\rrp{\mathopen{)\!)}}
\def\rrb{\mathopen{]\!]}}
\def\Coeff{\operatorname{Coeff}}

\newcommand{\grad}{\operatorname{grad}}

\newcommand{\spec}{\operatorname{Spec}}

\def\llp{\mathopen{(\!(}}
\def\llb{\mathopen{[\![}}
\def\rrp{\mathopen{)\!)}}
\def\rrb{\mathopen{]\!]}}

\newtheorem{thm}[subsection]{Theorem}
\newtheorem{lem}[subsection]
{Lemma}
\newtheorem{cor}[subsection]
{Corollary}
\newtheorem{prop}[subsection]
{Proposition}
\newtheorem{conj}
{Conjecture}
{Problem}

\theoremstyle{definition}
{Definition}
{Notation}
{Example}

\theoremstyle{remark}
\newtheorem{remark}[subsection]
{Remark}
\newtheorem{rem}[subsection]
{Remark}
{Remarks}

\theoremstyle{plain}

  {\par\medskip\noindent #1\par\begingroup%
    \advance\leftskip by 1em\advance\rightskip by 1em}%
  {\par\endgroup}

\DeclareMathOperator*{\Spec}{Spec}

\newcommand{\ord}{\operatorname{ord}}

\newcommand{\abs}[1]{\lvert#1\rvert}

\def\cO{\mathcal{O}}

\renewcommand{\phi}{\varphi}
\renewcommand{\epsilon}{\varepsilon}
\renewcommand{\theta}{\vartheta}

\renewcommand{\and}{ \quad \text{and} \quad }

\begin{document}

\setcounter{tocdepth}{1} 

\author[Raf Cluckers]
{Raf Cluckers}
\address{Univ.~Lille, CNRS, UMR 8524 - Laboratoire Paul Painlevé, F-59000 Lille, France, and,
KU Leuven, Department of Mathematics, B-3001 Leu\-ven, Bel\-gium}
\email{Raf.Cluckers@univ-lille.fr}
\urladdr{http://rcluckers.perso.math.cnrs.fr/}

\author[Kien~H.~Nguyen]
{Kien Huu Nguyen}
\address{KU Leuven, Department of Mathematics,
Celestijnenlaan 200B, B-3001 Leu\-ven, Bel\-gium, and, Normandie Université, Université de Caen Normandie - CNRS, Laboratoire de Mathématiques Nicolas Oresme (LMNO),UMR 6139, 14000 Caen, France, and,
Thang Long Institute of Mathematics and Applied Sciences, Hanoi, Vietnam
}
\email{kien.nguyenhuu@kuleuven.be, huu-kien.nguyen@unicaen.fr}
\urladdr{https://sites.google.com/site/nguyenkienmath/home}

\thanks{The authors would like to thank T.~Browning, T.~Cochrane, J.~Denef, N.~Katz, E.~Kowalski, F.~Loeser, M.~Musta\c{t}\u{a}, and J.~Wright for inspiring discussions on the topics of this paper. The authors 
are partially supported by the European Research Council under the European Community's Seventh Framework Programme (FP7/2007-2013) with ERC Grant Agreement nr. 615722
MOTMELSUM, by the Labex CEMPI  (ANR-11-LABX-0007-01), and by KU Leuven IF  C16/23/010. K.H.N. is partially supported by Fund for Scientific Research - Flanders (Belgium) (F.W.O.) 12X3519N, and by the Excellence Research Chair “FLCarPA: L-functions in positive characteristic and applications” financed by the Normandy Region. K.H.N. is also funded by Vingroup Joint Stock Company and supported by Vingroup Innovation Foundation (VinIF) under the project code VINIF.2021.DA00030,  and by Vietnam National Foundation for Science and Technology Development (NAFOSTED) under grant number 101.04-2019.316.}  
%
%
\subjclass[2020]{Primary 11L07; Secondary 11S40,  14E18, 03C98, 11F23}
\keywords{Bounds for exponential sums, Igusa's conjecture on exponential sums, Igusa's conjecture on monodromy, the strong monodromy conjecture, local-global principles, circle method, major arcs, minimal exponent, motivic oscillation index, log canonical threshold, Igusa's local zeta functions, motivic integration, $p$-adic integration, log resolutions}

\begin{abstract}
We combine two of Igusa's conjectures with recent semi-continuity results by Musta\c{t}\u{a} and Popa to form a new, natural conjecture about bounds for exponential sums. These bounds have a deceivingly simple and general formulation in terms of degrees and dimensions only. We provide evidence consisting partly of adaptations of already known results about Igusa's conjecture on exponential sums, but also some new evidence like for all polynomials in up to
 $4$ variables. We show that, in turn, these bounds imply consequences for Igusa's  (strong) monodromy conjecture.
 The bounds are related to estimates for major arcs appearing in the circle method for local-global principles. 
\end{abstract}

\title[Combining Igusa's conjectures with semi-continuity]{Combining Igusa's conjectures on exponential sums and monodromy with semi-continuity of the minimal exponent} 


\maketitle

\section{Introduction}\label{sec:intro}


Let $f$ be a polynomial in $n$ variables over $\ZZ$ and of degree $d>1$, and  let $s$ be the (complex affine) dimension of the critical locus of 
the degree $d$ homogeneous part 
of $f$.
The main objects of our study are the 
finite exponential sums from   (\ref{eq:fN}) and their estimates in terms of $n$, $d$, and $s$ as in Conjecture \ref{con1} below. 
For any positive integer $N$ and any complex primitive $N$-th root of unity $\xi$, consider the exponential sum
\begin{equation}\label{eq:fN}
 \sum_{x\in (\ZZ/N\ZZ)^n}  \xi^{f(x)}.
\end{equation}
When $N$ runs over the set of prime numbers, the sums from  (\ref{eq:fN}) fall under the scope of works by Grothendieck, Deligne, Katz, Laumon, and others,  building in particular on the Weil Conjectures. 
We don't pursue new results for $N$ running over the set of prime numbers. 
Instead we put forward new bounds for these sums uniformly in general $N$ with, roughly, a win of a factor
$$
N^{-(n-s)/d}
$$
on the trivial bound, see Conjecture \ref{con1} below. In this context, Birch \cite{{Birch}} proved and used bounds with exponent
$$
(n-s)/2^{d-1}(d-1)
$$
instead of our projected and stronger $(n-s)/d$, see (\ref{bound:Birch}) below.  The bounds in the conjecture look deceivingly simple, but a reduction argument to, say, the case $s=0$ turns out to be surprisingly hard in general, and moreover, the case $s=0$ for non-homogeneous $f$ is surprisingly hard as well.
As evidence for Conjecture \ref{con1} we prove an almost generic case (based on the Newton polyhedron of $f$), as well as the case with up to  $4$ variables, and, 
the cases restricted to those $N$ which are cube free and more generally 
$(d+2)$-th power free. 

Conjecture \ref{con1} combines two of Igusa's 
conjectures, namely on exponential sums  and on monodromy, and represents an update of these conjectures in line with the recently proved  semi-continuity result for the minimal exponent by Musta{\c{t}}{\v{a}} and Popa \cite{MustPopa} and the conjectured equality of the minimal exponent with the motivic oscillation index, see \cite{CMN} and Section \ref{subs:minimal:exp} below.  


\subsection{}
Let us make all this more precise, for $f$ a polynomial  over $\ZZ$ in $n$ variables.
For an integer $N>0$ and a complex primitive $N$-th root of unity $\xi$, put
\begin{equation}\label{eq:SfN}
E_f(N,\xi):= \abs{ \frac{1}{N^n }  \sum_{x\in (\ZZ/N\ZZ)^n}  \xi^{f(x)} },
\end{equation}
which is simply the complex modulus of the sum in (\ref{eq:fN}) normalized by the number of terms. 
Write $d$ for the degree of $f$ and $f_d$ for the homogeneous degree $d$ part of $f$. We assume that $d>1$.
Write $s=s(f)$ for the dimension of the critical 
locus of $f_d$, namely, of the solution set in $\CC^n$ of the equations
\begin{equation}\label{eq:grad-fd}
0 = 
\frac{\partial f_d}{ \partial x_1}(x)  =\ldots = \frac{\partial f_d}{\partial x_n}(x).
\end{equation}
Note that $0\leq s\le n-1$. Our projected bounds are as follows : 

\begin{conj}\label{con:SfN}\label{con1}
Given $f$, $n$, $s$, and $d$ as above and any $\varepsilon>0$, one has
\begin{equation}\label{bound:SfN}
E_f(N,\xi) \ll N^{-\frac{n-s}{d}+\varepsilon}.
\end{equation}
\end{conj}

%
%


In this context, note that Birch \cite[Lemma 5.4]{Birch} obtained the following bound, based on the very same data of $f,n,s,$ and $d$ (and, assuming $f$ to be homogeneous):
\begin{equation}\label{bound:Birch}
E_f(N,\xi) \ll N^{-\frac{n-s}{2^{d-1}(d-1)}+\varepsilon},
\end{equation}
which he used to estimate major arcs to obtain general logal-global principles (see Section \ref{sec:intro:locglob} below).

Remarkably, the weakening of Conjecture \ref{con1} with
\begin{equation}\label{bound:Kien}
(n-s)/2(d-1)+\varepsilon
\end{equation}
in the exponent of $N$ in (\ref{bound:SfN}) instead of $(n-s)/d+\varepsilon$ has just been shown in \cite{Nguyen:n/2d}, vastly improving Birch's bounds (\ref{bound:Birch}).
The case of Conjecture \ref{con1} with $d=3$ is in line with the resembling (but averaged) bounds (170) of \cite{Hooley-non}. In the one variable case, similar bounds as in (\ref{bound:SfN}) have already been studied, see e.g.~\cite{Chalk}, \cite{Hua1959} and some generalisations in \cite{CochZheng1, CochZheng2}. Knowing only $n$, $s$ and $d$, the exponent $-(n-s)/d$ is optimal in (\ref{bound:SfN}), as witnessed by $f=\sum_{i=1}^{n-s} x_i^d$ and $N=p^d$ for primes $p$, see also the example in  (\ref{eq:igusa:ex}).


\begin{remark}\label{rem:not}
The notation in (\ref{bound:SfN}) 
means that, given $f$ and $\varepsilon>0$, there is a constant $c=c(f,\varepsilon)$ 
such that, for all integers $N\geq 1$ and all primitive $N$-th roots of unity $\xi$, the value $E_f(N,\xi)$ is no larger than 
$cN^{- \frac{n-s}{d} + \varepsilon}$.
\end{remark}

\begin{remark}\label{rem:chinese}
The critical case of Conjecture \ref{con1} is with $N$ having a single prime divisor. 
Indeed, by the Chinese Remainder Theorem, if one writes $N= \prod_{i}p_i^{e_i}$ for distinct prime numbers $p_i$ and integers $e_i>0$, then one has
\begin{equation}\label{eq:p-i}
E_{f}(N,\xi) = \prod_{i} E_f(p_i^{e_i},\xi_i)
\end{equation}
for some primitive $p_i^{e_i}$-th roots of unity $\xi_i$. In detail, if one writes $1/N=\sum_i a_i/p_i^{e_i}$ with $(a_i,p_i)=1$, then one takes $\xi_i = \xi^{b_i}$ with $b_i=a_i N/p_i^{e_i}$. 
\end{remark}

\subsection{}\label{sec:introbeyond}
Conjecture \ref{con1} simplifies Igusa's original question on exponential sums (recalled in Section \ref{sec:intro:original}) to bounds involving only $n$, $d$, and $s$.
It opens a way to proceed with Igusa's conjecture on exponential sums beyond the case of non-rational singularities that is obtained recently in \cite{CMN}. 


In most of the evidence that we provide below,  one can furthermore take $\varepsilon=0$ and
one may wonder to which extent this sharpening of Conjecture \ref{con:SfN} holds. Such a sharpening with $\varepsilon=0$ goes beyond Igusa's conjectures in ways explained in Section \ref{subsection:order:poles}. One may also wonder whether the implied constant $c$ can be taken depending only on $f_d$ and $n$ (and $\varepsilon$, but not on $f$). If one excludes a finite set $S$ (depending on $f$ or just on $f_d$) of prime divisors of $N$, then it seems furthermore possible that the implied constant can be taken depending only on $d$ and $n$  (and $\varepsilon$); see Remark \ref{rem:largeprimes-n-d} for more details.

\subsection{}\label{sec:intro:monod}
In Section \ref{sec:monodromy} we relate the bounds from Conjecture \ref{con1} to Igusa's monodromy conjecture.
Conjecture \ref{con1} implies the strong monodromy conjecture for poles of local zeta functions with real part in the range strictly between $-\frac{n-s}{d}$ and zero. More precisely, we show under Conjecture \ref{con1} that  there are no poles (of a local zeta function of $f$) with real part in this range except $-1$ (see Proposition \ref{prop:mon}); from \cite{MustPopa} it follows correspondingly that there are no zeros of the Bernstein-Sato polynomial of $f$ in this range other than $-1$ (see Proposition \ref{lem:alpha}). Note that Conjecture \ref{con1} is much stronger than the strong monodromy conjecture in the mentioned range, as the latter implies merely a much weaker variant of Conjecture \ref{con1}, namely the bounds from (\ref{bound:SfNp}) instead of (\ref{bound:SfN}), where the constant $c_p$ is allowed to depend on $p$.

\subsection{}\label{sec:intro:original}
\change{Igusa's original question on exponential sums predicts upper bounds with a non-canonical exponent coming from a choice of log resolution for homogeneous $f$ (with $f=f_d$), see 
\cite{Igusa3}.  More precisely, let $h: Y\to X=\AA_\QQ^n$ be a log resolution of $D=f^{-1}(0):=\Spec(\QQ[x_1,...,x_n]/(f))$, i.e. $Y$ is an integral smooth scheme, $h:Y\rightarrow X$ is a proper map, the restriction $h:Y\backslash h^{-1}(D)\rightarrow X\backslash  D$ is an isomorphism, and $(h^{-1}(D))_{\text{red}}$ has simple normal crossings as subscheme of $Y$. Such a log-resolution exists by the work of Hironaka \cite{Hir:Res}. Write $h^{-1}(D)=\sum_{i\in I}N_iE_i$ and $\Div(h^*(dx_1\wedge...\wedge dx_n))=\sum_{i\in I}(\nu_i-1)E_i$ for irreducible components $E_i$ of $(h^{-1}(D))_{\text{red}}$ and positive integers $N_i,\nu_i$.  By blowing up further, one may suppose that $E_j\cap E_i= \emptyset$ whenever $(\nu_i, N_i)=(\nu_j,N_j)=(1,1)$ and $i\neq j$. Put
$$
J=\{i\in I|(\nu_i, N_i)\neq (1,1)\} 
$$
and
$$
\sigma_0=\sigma_0(h):=\min_{i\in J}\frac{\nu_i}{N_i}.
$$
 Note that $\sigma_0$ depends on the choice of $h$ in general. Igusa originally conjectured, for any  $\sigma<\sigma_0$,  and under a few extra conditions that are most likely superfluous (namely, that $\sigma_0>2$ and that $f$ is homogeneous), that one has a bound
\begin{equation}\label{Iguori}
E_f(N,\xi)\leq cN^{-\sigma}
\end{equation}
for all $N>0$, all primitive $N$-th roots of unity  $\xi$ and a constant $c$ independent of $N,\xi$. In the case that $\sigma_0\le 1$, the bounds (\ref{Iguori}) are proved even more generally than in Igusa's original conjecture in \cite{CMN}, see Section \ref{sec:ev2} below.  Furthermore, precisely (and only) in the case that $\sigma_0\le 1$ holds, the value $\sigma_0$ is independent of the choice of $h$, and, is called the log-canonical threshold of $f$.

\par

When one takes a fixed prime number $p$, Igusa \cite{Igusa3} proves that Inequality (\ref{Iguori}) holds for $N$ of the form $N=p^m$ with $m\geq 1$, primitive $N$-th root of unity $\xi$ and a constant $c=c_p$ depending on $p$ and $\sigma<\sigma_0$ (but not on $m,\xi$).

\par

When $f$ satisfies the non-degeneracy  condition of Section \ref{sec:CAN}, there is a toric log-resolution $h$ of $D$ related to the Newton polyhedron of $f$ at zero. In this case, $\sigma_0(h)=\sigma_f$ with $\sigma_f$ defined again from the Newton polyhedron (see Section \ref{sec:CAN} for the definition of $\sigma_f$). In \cite{DenSper}, Denef and Sperber conjectured that when $f$ is non-degenerate, one can replace Inequality (\ref{Iguori}) in Igusa's conjecture by
\begin{equation}\label{Denef-Spe}
E_f(p^m,\xi)\leq cm^{\kappa-1}p^{-\sigma_f m},
\end{equation}
where $\kappa$ is an invariant coming from the Newton polyhedron of $f$ at zero (see Section \ref{sec:CAN}) and $c$ is independent of $p,m, \xi$. Thus, the Denef-Sperber conjecture is a bit stronger than Igusa's conjecture in the case of non-degenerate polynomials, by the more explicit form of the exponents $\sigma_f$ and  $\kappa$; it has been proved and generalized in \cite{DenSper,CDenSper,CDenSperlocal,CAN}, see Proposition \ref{prop:CANsigma} below.}

 In \cite{CVeys}, some of  Igusa's original conditions, like homogeneity for $f$, were dropped with some care, namely by focusing on squareful integers, see 
Section \ref{subs:minimal:exp} 
 for more details. Igusa's condition that $\sigma_0(h)>2$ for some $h$ was already dropped before; it was more relevant for his intended application of his conjecture to local-global principles than for the content of the conjecture itself. \change{Additionally, Igusa's non-canonical exponent $\sigma_0(h)$} was replaced by canonical  candidates for the exponent: the motivic oscillation index of $f$, and, (expected to be equal) the minimal exponent of $f-v$ with $v$ a well-chosen critical value of $f:\CC^n\to\CC$, see \cite{CVeys}, \cite{CMN} and Section \ref{subs:minimal:exp} below.
Our suggested bounds encompass several issues related to the minimal exponent (and, the motivic oscillation index), 
by replacing them by the much simpler and natural value $\frac{n-s}{d}$, yielding Conjecture \ref{con1}  as new variant of (\ref{Iguori}). As an extra upshot, Conjecture \ref{con1} makes sense again for all positive integers $N$, and not only for squareful integers. 
Although the bounds from Conjecture \ref{con1} seem simple and very natural, they appear surprisingly hard to show in general, and even the much weaker bounds with constants depending on $p$ and $N$ running over powers of $p$ as in (\ref{bound:SfNp}), remain elusive in general up to date, even in the case with $s=0$.

\subsection{}
From Section \ref{sec:intro:evidence} on we develop evidence for Conjecture \ref{con1}. We first rephrase some well-known results as evidence, 
namely, Igusa's treatment of  the smooth homogeneous case (with $f=f_d$ and $s=0$), the case with degree $d=2$, the case with $(n-s)/d\le 1$, the case of at most $3$ variables, and, the case with cube free $N$.  We then generalize this further to new evidence for all $N$ which are $(d+2)$-th power free (see Section \ref{sec:HeathB}). This treatment of the $(d+2)$-th power free case is mainly provided for expository reasons, as it uses some recent results on bounds of \cite{CGH4} in the context of motivic integration and uniform $p$-adic integration as in \cite{CGH5}; it indicates that the case  $N=p^e$ with $p$ prime and $e$ small is generally more easy than with $e$ large. 

In Section \ref{sec:CAN}, we show Conjecture \ref{con:SfN} when $f$ is non-degenerate with respect to its Newton polyhedron at zero, using
recent work from \cite{CAN} and some elementary reasonings on Newton polyhedrons. This shows that Conjecture \ref{con1} holds under often generic conditions, including the generic weighted homogeneous case, 
see Remark \ref{rem:gen}.

In Section \ref{sec:n=4}, we show Conjecture \ref{con1} for all polynomial in up to $4$ variables. This uses \cite{CMN} to reduce to the case with $n=4$, $d=3$ and $s=0$. To finish this degree three case in four variables we use Weierstrass preparation and the results from \cite{CMN} and \cite{MustPopa}.

In our final Remark \ref{rem:largeprimes-n-d}, we explain that throughout the evidence for Conjecture \ref{con1} of this paper, up to excluding a finite set $S$ of primes divisors of $N$ (depending on $f$), the constant $c$ can be taken depending only on $d$, $n$ and $\varepsilon$.

Let us finally mention the further evidence of \cite{Nguyen:n/2d} for Conjecture \ref{con1}, with the weakened exponent $(n-s)/2(d-1)$ in the upper bound of (\ref{bound:SfN}) instead of $(n-s)/d$.

\subsection{}\label{sec:intro:vast}
In his vast program from \cite{Igusa3}, Igusa studies a certain ad\`elic Poisson summation formula related to $f$, inspired by Weil's work \cite{Weil} on the Hasse principle and Birch's work  \cite{Birch} on more general local-global principles.
Conjecture \ref{con:SfN}  would imply that Igusa's ad\`elic Poisson summation formula  for $f$ holds under the simple condition
\begin{equation}\label{bound:2d}
 n -s > 2d
 \end{equation}
which simplifies (and generalizes) the list of conditions put forward by Igusa in \cite{Igusa3}, and would drop in particular the condition of homogeneity on $f$.


\subsection{}\label{sec:intro:locglob}
Also for obtaining (or just for streamlining) local-global principles, Conjecture \ref{con:SfN} may play a role. 
When $f$ is homogeneous, the sums $E_f(N,\xi)$ appear for estimating the contribution of the major arcs in the circle method 
to get a local-global principle for $f$ when
\begin{equation}\label{bound:2^d}
n-s > (d-1) 2^d,
\end{equation}
in work by Birch \cite{Birch} and in the recent sharpening from \cite{BrowPren}, which both rely on Birch's bounds (\ref{bound:Birch}) quoted above. Birch shows that any homogeneous form $f=f_d$ with (\ref{bound:2^d}) and having smooth local zeros for each completion of $\QQ$ automatically has a nontrivial rational zero. One may hope one day to replace Condition (\ref{bound:2^d}) on homogeneous $f$ by (\ref{bound:2d}), which is in line with a conjectured local-global principle from \cite{Browning-HB}.
Conjecture \ref{con:SfN} would put the remaining obstacle 
completely with the estimation for the minor arcs (where actually already lie the limits of the current strategies).
Other possible applications may be for small solutions of congruences as studied in e.g.~\cite{Baker}.


%
Note that Birch's method in \cite{Birch} also helps to understand the distribution of rational points in the projective hypersurface $X$ associated with a homogeneous polynomial $f$. More precisely, the singular series
$$\mathfrak{S}(f)=\sum_{N\geq 1}N^{-n}\sum_{a\in (\ZZ/N\ZZ)^{\times}}\sum_{x\in (\ZZ/N\ZZ)^n}\exp(\frac{2\pi iaf(x)}{N})$$
 which is equal to the  product of $p$-adic densities of $f$ will contribute to the dominant term in the asymptotic formula of the number points of $X$ of bounded height. Conjecture \ref{con:SfN} implies that the singular series $\mathfrak{S}(f)$ is absolutely convergent if $n-s>2d$. Thereby Conjecture \ref{con:SfN} may be useful for the future research on the distribution of rational points in algebraic hypersurfaces.
\subsection{Generalization to a ring of integers}
Before giving precise statements and proofs, we formulate 
a natural generalization to rings of integers (a generality we will not use later in this paper).
For a ring of integers $\cO$ of a number field and a polynomial $g$ over $\cO$, one can formulate an analogous conjecture with summation sets $(\cO/I)^n$ with nonzero ideals $I$ of $\cO$ and primitive additive characters $\psi:\cO/I\to\CC^{\times}$. More precisely, let $g$ be a polynomial in $n$ variables of degree $d>1$ and with coefficients in $\cO$. For any nonzero ideal $I$ of $\cO$ and any primitive additive character $\psi:\cO/I\to\CC^{\times}$, let $N_I:=[\cO:I]$ be the absolute norm of $I$ and consider 
\begin{equation}\label{eq:SfI}
E_g(I,\psi):= \abs{ \frac{1}{N_I^n }  \sum_{x\in (\cO/I)^n}  \psi(g(x)) }.
\end{equation}
Write $s$ for the dimension of the critical 
locus of the degree $d$ homogeneous part $g_d$ of $g$.  As a generalization of the above questions, one may wonder whether
for each $\varepsilon>0$ (or more strongly with $\varepsilon=0$) one has
\begin{equation}\label{bound:SgN}
E_g(I,\psi) \ll N_I^{-\frac{n-s}{d}+\varepsilon}.
\end{equation}
As above with the Chinese remainder theorem, one can rephrase this using the finite completions of the field of fractions of $\cO$. Furthermore, one can study similar sums for the local fields $\FF_q \llp t \rrp$ (with similar methods in the large characteristic case), see e.g.~\cite[Section 2.6]{CVeys} and \cite[Section 1.2]{CMN}.






\section{Link with the monodromy conjecture}\label{sec:monodromy}

\subsection{}\label{sec:mon}
Fix a prime number $p$. For each integer $m\ge 0$ let $a_{p,m}$ be the number of solutions in $(\ZZ/p^m\ZZ)^n$ of the equation $f(x)\equiv 0\bmod p^m$, and consider the Poincar\'e series
$$
P_{f,p}(T):= \sum_{m\ge 0} \frac{a_{p,m}}{p^{mn}} T^m,
$$
in $\ZZ[[T]]$. \change{Igusa \cite{Igusa2,Igusa3} showed that $P_{f,p}(T)$ is a rational function in $T$, using a log resolution of $f^{-1}(0)$. Let $T_0$ be a complex pole of $P_{f,p}(T)$ and let $t_0$ be the real part of a complex number $s_0$ with $p^{-s_0}=T_0$. Let $h:Y\to \AA_\QQ^n$ be a log-resolution  of $f^{-1}(0)$ and $(N_i, \nu_i)_{i\in I}$ as in Section \ref{sec:intro:evidence}. Igusa \cite{Igusa3} showed that $t_0$ belongs to the set $\mathcal{P}_h=\{-\nu_i/N_i|i\in I\}$. However, $\mathcal{P}_h$ depends on the choice of log-resolution $h$.  Igusa \cite[Theorem 2]{Igusa2} also showed a strong link between exponential sums and local zeta functions  (see also \cite[Corollary 1.4.5]{DenefBour} and \cite[Proposition 2.7]{DenefVeys}), yielding the following corollary.
\begin{cor}[\cite{Igusa2}]\label{cor:Igusa2thm2}
For $p$, $f$, $T_0$ and $t_0$ as above, if $T_0$ is furthermore a pole of $(T-1/p)P_{f,p}(T)$, then
\begin{equation}\label{lowerboun.0}
p^{m t_0}\leq c'_p E_f(p^m,\xi)
\end{equation}
for infinitely many  $m$ and $\xi$ and a constant $c'_p$ independent of $m$, $\xi$.
\end{cor}
\begin{proof}
Proposition 2.7 of \cite{DenefVeys} gives finitely many complex numbers $T$, finitely many characters $\chi:\CC^\times\to\CC^\times$ of finite order, finitely many integers $b\ge 0$, and finitely many complex numbers $c$ such that for large $m$,  $E_f(p^m,\xi)$ is (the complex modulus of) a finite $\CC$-linear combination of the terms
$$
\chi(\xi)\cdot T^m m^{b}\xi^c ,
$$
where furthermore a term with $T=T_0$ appears non-trivially in this linear combination for each pole $T_0$ of $(T-p^{-1})P_{f,p}(T)$. Now the corollary follows by looking at the dominant terms, namely, with largest occurring real part of $T$ and for such $T$ the largest occurring value for $b$.
\end{proof}
}  Denef \cite{DenefBour} formulated a strong variant of Igusa's monodromy conjecture by asking whether $t_0$ as above is automatically a zero of the Bernstein-Sato polynomial of $f$. The following result addresses this question in a range of values for $t_0$, namely strictly between $-(n-s)/d$ and zero, assuming Conjecture \ref{con1} for $f$.

\begin{prop}[Strong Monodromy Conjecture, in a range]\label{prop:mon}
Let $f$, $n$, $s,$ and $d$ be as in the introduction and suppose that  Conjecture \ref{con:SfN} holds for $f$. Let $t_0$ be coming as above from a pole $T_0$ of $P_{f,p}(T)$ for a prime number $p$. Suppose that moreover $t_0>-(n-s)/d$. Then $t_0=-1$, and hence, $t_0$ is a zero of the Bernstein-Sato polynomial of $f$. 
\end{prop}
Proposition \ref{prop:mon} is a form of the strong monodromy conjecture in the range strictly between $-(n-s)/d$ and zero. We don't pursue the highest generality here, and leave the generalization for other variants of zeta functions like twisted $p$-adic local zeta (or even motivic) functions to the reader.
Proposition \ref{lem:alpha} below gives a related statement for the zeros of the Bernstein-Sato polynomial of $f$.

\begin{proof}[Proof of Proposition \ref{prop:mon}]
Let $p$ be a prime number. Let $t_0$ be the real part of a complex number $s_0$ such that $T_0:=p^{-s_0}$ is a pole of $P_{f,p}(T)$.
Suppose that  
for all $\varepsilon>0$ there exists $c_p=c_p(f,\varepsilon)$ such that
\begin{equation}\label{bound:SfNp}
E_f(p^m,\xi) \leq c_p\cdot  (p^m)^{-\frac{n-s}{d}+\varepsilon} \mbox{ for all  $m>0$ and all primitive $\xi$}.
\end{equation}
By Corollary \ref{cor:Igusa2thm2} it follows that $t_0$ either equals $-1$ \change{or, one has
\begin{equation}\label{lowerboun}
p^{mt_0}\leq c'_pE_f(p^m,\xi)
\end{equation}
for infinitely many pairs $(m,\xi)$ and a constant $c'_p$ independent of $m,\xi$.
Clearly the bound from (\ref{bound:SfNp}) holds if Conjecture \ref{con1} holds for $f$. By (\ref{bound:SfNp}) and (\ref{lowerboun}),}  if $t_0>-(n-s)/d$, then $t_0=-1$. \change{Since $f$ is non-constant}, the value $-1$ is automatically a zero of the Bernstein-Sato polynomial of $f$. This completes the proof of the proposition.
\end{proof}
Showing the bounds (\ref{bound:SfNp}) from the above proof for general $f$ 
does not seem easy, although they are much weaker (and much less useful adelically) than the bounds from (\ref{bound:SfN}), because of the dependence of $c_p$ on $p$.

In view of the strong monodromy conjecture, Proposition \ref{prop:mon} should be compared with the following absence of zeros of the Bernstein-Sato polynomial in a similar range, apart from $-1$. Recall that the zeros of the Bernstein-Sato polynomial are negative rational numbers.
\begin{prop}\label{lem:alpha}
Let $f$, $n$, $s,$ and $d$ be as in the introduction and let $r$ be any zero of the Bernstein-Sato polynomial of $f$. Then either $r=-1$, or, $r \leq - (n-s) /d$. 
\end{prop}
\begin{proof}
\change{We write $f=f_0+...+f_d$ with $f_i$ is the homogeneous part of degree $i$ of $f$. Item (3) of Theorem~E of \cite{MustPopa} states that the minimal exponent $\tilde{\alpha}_{f,0}$ of $f$ is at least $(n-s)/d$ if $f$ is homogeneous. Recall that the minimal exponent $\tilde{\alpha}_f$ of $f$ is equal to $\min_{x\in f^{-1}(0)}\tilde{\alpha}_{f,x}$. Moreover, if $\phi:\AA^n\to \AA^n$ is a linear change of variables then $\tilde{\alpha}_{f,x}=\tilde{\alpha}_{f\circ \phi, \phi^{-1}(x)}$, and, for any constant $\beta\neq 0$ one has $\tilde{\alpha}_{f,x}=\tilde{\alpha}_{\beta f,x}$. Let $g_\lambda(x)$ be $f_d+\sum_{0\leq i\leq d-1}\lambda^{d-i}f_i$.  Then for each $\lambda\neq 0$ we have  $g_\lambda(x)=\lambda^df(x/\lambda)$.  Write $X=\AA^n\times \AA^1$, $T=\AA^1$, $\pi:\AA^n\times \AA^1\to \AA^1$ for the projection, $h(x,\lambda)=g_\lambda(x)$ and $D=h^{-1}(0)$.  For each $x\in f^{-1}(0)$, we consider the section $s_x: T\to X$ with $ \lambda\mapsto (\lambda x, \lambda)$, then $s_x(\lambda)\in D_\lambda$ since $h(\lambda x, \lambda)=g_\lambda(\lambda x)=\lambda^df(\lambda x/\lambda)=0$ if $\lambda\neq 0$ and $h(0,0)=f_d(0)=0$. Now we can use item (2) of Theorem~E of \cite{MustPopa} for $X$, $T$, $\pi$, $D$ and $s_x$  to see that  for each $x\in f^{-1}(0)$ we have
$$
\tilde{\alpha}_{f,x}=\tilde{\alpha}_{g_\lambda, \lambda x}\geq \tilde{\alpha}_{g_0,0}=\tilde{\alpha}_{f_d,0}\geq (n-s)/d
$$
for all  $\lambda\neq 0$ in a small enough neighbourhood of $0$. Thus,
$$
\tilde{\alpha}_f=\min_{x\in f^{-1}(0)}\tilde{\alpha}_{f,x}\geq (n-s)/d.
$$
The proposition now follows directly from 
the definition of the minimal exponent $\tilde{\alpha}_f$ of $f$ as the smallest zero of $b_f(-s)/(s-1)$, where $b_f(s)$ is the Bernstein-Sato  polynomial of $f$.} 
\end{proof}


\subsection{}\label{subsection:order:poles}
The variant of Conjecture \ref{con:SfN}  with $\varepsilon=0$ (or even just the bounds (\ref{bound:SfNp}) with $\varepsilon=0$) implies for any pole $T_0$ of $P_{f,p}(T)$ with corresponding value $t_0$ the following bound on the order of the pole : If $t_0$ equals  $-(n-s)/d$ and  $-(n-s)/d\neq -1$, then the pole $T_0$ has multiplicity at most one, and, if $t_0=-1= -(n-s)/d$, then the pole $T_0$ has multiplicity at most two, by a similar reasoning as for Corollary \ref{cor:Igusa2thm2}.

\begin{rem}\label{rem:(5)}
Conjecture \ref{con:SfN} implies the bounds (\ref{bound:SfNp}) with moreover constants $c_p$ taken independently from $p$, and, the conjecture in turns would follow from this.
By (\ref{eq:p-i}), the variant of Conjecture \ref{con:SfN} with $\varepsilon=0$ is equivalent with the bounds (\ref{bound:SfNp}) with $\varepsilon=0$ and such that furthermore the products of the constants $c_p$ over any set $P$ of primes is bounded independently of $P$.
 \end{rem}

\subsection{}\label{subs:minimal:exp}
The minimal exponent of $f$ is defined as the \change{smallest} zero of the quotient $b_f(-s)/(s-1)$ with $b_f(s)$ the Bernstein-Sato polynomial of $f$ if such a zero exists, and it is defined as $+\infty$ otherwise. Write $\hat \alpha_f$ for the minimum of the minimal exponents of $f-v$ for  $v$ running over the (complex) critical values of $f$.
In a more canonical variant of Igusa's original question, one may wonder more technically than Conjecture \ref{con1} whether
for all $\varepsilon>0$ one has
\begin{equation}\label{bound:SfNfull}
E_f(N,\xi) \ll N^{-\hat\alpha_f+\varepsilon} \mbox{ for all $\xi$ and all squareful integers $N$},
\end{equation}
similarly as the question introduced in \cite{CVeys} for the motivic oscillation index (and where the necessity of working with squareful integers $N$ is explained). Recall that an integer $N$ is called squareful if for any prime $p$ dividing $N$ also $p^2$ divides $N$. In \cite{CAN}, \cite{Saskia-Kien}, \cite{CDenSper}, \cite{CDenSperlocal}, \cite{CMN}, \cite{DenSper}, evidence is given for this sharper but more technical question. As mentioned above, $\hat\alpha_f$ is hard to compute in general, and $(n-s)/d$ is much more transparent. However, $\hat\alpha_f$ is supposedly equal to the motivic oscillation index of $f$, which in turn is optimal as exponent of $N^{-1}$ in the upper bounds for $E_f(N,\xi)$ for squareful  $N$ (see the last section of \cite{CMN}, or, a reasoning as for Corollary \ref{cor:Igusa2thm2}). Note that by Proposition \ref{lem:alpha}, one has
\begin{equation}\label{bound:min:exp}
\hat\alpha_f \ge (n-s)/d,
\end{equation}
which shows that (\ref{bound:SfNfull}) is indeed a sharper (or equally sharp) bound than (\ref{bound:SfN}).

\section{Some first evidence}\label{sec:intro:evidence}

In this section we translate some well-known results into evidence for Conjecture \ref{con1}, and we show the (new) case of $(d+2)$-th power free $N$. A key (but hard) case of Conjecture \ref{con1} for inhomogeneous $f$ is when $f_d$ is projectively smooth, namely with $s=0$, since the case of general $s$ can be derived from a sufficiently uniform form of the inhomogeneous case with $s=0$, see e.g.~how (\ref{eq:deligne}) is used below for squarefree $N$.  However, the inhomogeneous case with $s=0$ seems very hard at the moment. This should not be confused with Igusa's more basic case recalled in Section \ref{sec:ev1}, for homogeneous $f$ with $s=0$.

\subsection{}\label{sec:ev1}
When $f$ itself is smooth homogeneous, namely, $f=f_d$  and $s=0$, then Conjecture \ref{con:SfN} with $\varepsilon=0$ is known by Igusa's bounds from \cite{Igusa3}, by a straightforward computation and reduction to Deligne's bounds.
In detail,  if $f=f_d$ and $s=0$, Igusa \cite{Igusa3} showed (using \cite{DeligneWI}) that for each prime $p$ there is a constant $c_p$ such that
\begin{equation}\label{eq:igusa}
E_{f}(p^{m},\xi) \leq c_p p^{-mn/d} \mbox{ for all integers $m>0$ and all choices of $\xi$},
\end{equation}
and, that one can take $c_p=1$ when $p$ is larger than some value $M$ depending on $f$. More precisely, one can take $c_p=1$ when $p$ does not divide $d$ and when the reduction of $f$ modulo $p$ is smooth.
Furthermore, Igusa \cite{Igusa3}  shows that the exponent $-n/d$ of $p^m$ is optimal in the upper bound of (\ref{eq:igusa}) when $m=d$.
This easily shows that the exponent $(n-s)/d$ is optimal in Conjecture \ref{con1}, for example by taking 
\begin{equation}\label{eq:igusa:ex}
f = (x_1+\ldots+x_{s+1})^d + x_{s+2}^d + \ldots + x_n^d
\end{equation}
and $N=p^d$ for all prime numbers $p$.

\subsection{}\label{sec:ev2}
When $f$ is such that
\begin{equation}\label{eq:leq1}
(n-s)/d \leq 1,
\end{equation}
then Conjecture \ref{con:SfN} follows from \cite{CMN} and its recent solution of Igusa's conjecture for non-rational singularities. Indeed, in \cite{CMN} the stronger (and optimal) upper bounds from (\ref{bound:SfNfull}) are shown for all squareful $N$ in the case of non-rational singularities, as well as the case with $1$ in the exponent instead of $\hat\alpha_f$ in the case of rational singularities.   Recall that this is indeed stronger, by  (\ref{bound:min:exp}). The bounds for those integers $N$ that are not squareful are recovered by the treatment of squarefree $N$ below, by writing a general integer as a product of a squareful and a squarefree integer. We mention on the side that $\hat\alpha_f \le 1$ if and only if $f-v=0$ has non-rational singularities for some critical value $v\in \CC$ of $f$, by \cite{Saito-rational} and that in this case $\hat\alpha_f$ equals the minimum of the log canonical thresholds of $f-v$ for  $v$ running over the (complex) critical values of $f$.
\change{These results under condition (\ref{eq:leq1}) imply  that} Conjecture \ref{con:SfN} holds for all $f$ in three (or less) variables. Indeed, the degree two case is easy by diagonalizing $f_2$ over $\QQ$, and, (\ref{eq:leq1}) holds when $n\le 3\le d$. More surprizingly, Igusa's Conjecture (with the motivic oscillation index in the upper bound) is proved recently in \cite{NguyenVeys} for all polynomials in $3$ variables. Some related results of the special case with $n\le 2$ are developed in \cite{Fraser-Wright}, \cite{Lichtin4}, \cite{Veys:powerC}. In  Section \ref{sec:n=4} we will prove that Conjecture \ref{con:SfN} holds for all polynomials in up to four variables.

\subsection{}\label{sec:ev2.bis}
Although it is classical, let us explain the case of $d=2$ in more detail, by showing that Conjecture \ref{con1} holds with  $\varepsilon=0$ for $f$ of degree $d=2$. In fact, the argument as in the proof of Lemma 25 of \cite{Heath-Brown-d-2} is shorter and simpler for the case $d=2$, but our treatment will be useful later in this paper.
First suppose that the degree two part of $f$ is a diagonal form, namely, $f_2(x)=\sum_{i=1}^n a_i x_i^2$ for some $a_i\in \ZZ$.
In this case it is sufficient to show the case with $n=1$ and $d=2$  (indeed, $f=h_1(x_1)+\ldots+h_n(x_n)$ for some polynomials $h_j$ in one variable $x_j$ and of degree $\le 2$). But this case follows readily from Hua's bounds, see \cite{Hua1959} or \cite{Chalk} and is in fact elementary. 

For general $f$ with $d=2$, by diagonalizing $f_2$ over $\QQ$ and taking a suitable integer multiple, we find a matrix $T\in\ZZ^{n\times n}$ with nonzero determinant so that $f_2(Tx)$ is a diagonal form over $\ZZ$ in the variables $x$, namely, $f_2(Tx)=\sum_{i=1}^n a_i x_i^2$ for some $a_i\in \ZZ$. The map sending $x$ to $Tx$ transforms $\ZZ_p^n$ into a set of the form $\prod_{j=1}^n p^{e_{p,j}}\ZZ_p$ for some integers $e_{p,j}\ge 0$ (called a box). By composing with a map of the form $(x_j)_j\mapsto (b_jx_j)_j$ for some integers $b_j$ it is clear that we may assume that $T$ is already such that $e_{p,j}=e_p$ for all $p$ and all $j$ and some integers $e_p\ge 0$. Hence, the case $d=2$ follows from Lemma \ref{lem:T}.

\begin{lem}\label{lem:T}
Let $f$, $n$, $s,$ and $d$ be as in the introduction and let $T\in\ZZ^{n\times n}$ be a matrix with nonzero determinant and such that, for each prime $p$, the transformation $x\mapsto Tx$ maps $\ZZ_p^n$ onto $p^{e_p}\ZZ_p^n$ for some $e_p\ge 0$. Then, Conjecture \ref{con1} for each of the polynomials $g_i(x) := f(i+Tx)$ for $i\in  \ZZ^n$ implies conjecture \ref{con1} for $f$, and, similarly for Conjecture \ref{con1} with $\varepsilon=0$.
\end{lem}
\begin{proof}
For each $i\in\ZZ^n$, write $g_i(x)$ for the polynomial $f(i+ Tx)$.
For any prime $p$, let $\mu_{p,n}$ be the Haar measure on $\QQ_p^n$, normalized so that $\ZZ_p^n$ has measure $1$. For any integer $m>0$ and  any primitive $p^m$-th root of unity $\xi$, we have, by the change of variables formula for $p$-adic integrals, and with $e=e_p$ and with integrals taken against the measure  $\mu=\mu_{p,n}$, 
\begin{eqnarray*}
E_f(p^m,\xi) &   = &   \abs{ \int_{x\in \ZZ_p^n} \xi^{f(x)\bmod p^m} \mu } \\
& \le & \sum_{j=1}^n\sum_{i_j=0}^{p^{e}-1} \abs{\int_{x\in i +(p^{e}\ZZ_p)^n} \xi^{f(x)\bmod p^m} \mu }. \\
\end{eqnarray*}
For each $i$ we further have
\begin{eqnarray*}
\abs{\int_{x\in i +(p^{e}\ZZ_p)^n} \xi^{f(x)\bmod p^m} \mu } & =&  p^{-ne} \abs{\int_{x\in \ZZ_p^n} \xi^{f(i+p^{e}x)\bmod p^m} \mu }\\
&=& p^{-ne}     \abs{\int_{x\in \ZZ_p^n} \xi^{g_i (x)\bmod p^m} \mu } \\
&   = &  p^{-ne}E_{g_i}(p^m,\xi).
\end{eqnarray*}
Since $e_p=0$ for all but finitely many primes $p$, we are done.
 \end{proof}

\subsection{}\label{sec:ev3}
When one restricts to integers $N$ which are squarefree (namely, not divisible by a nontrivial square), then Conjecture \ref{con:SfN} with $\varepsilon=0$ follows from Deligne's bound from \cite{DeligneWI}, as we now explain. The reasoning is classical but also instructive for later use in this paper; a similar induction argument on $s\ge 0$ already appears in \cite{Hooley.int}.
By \cite{DeligneWI}, for each prime number $p$ such that the reduction of $f_d$ modulo $p$ is smooth, one has
\begin{equation}\label{eq:deligne}
E_{f}(p,\xi) \leq (d-1)^n p^{-n/2}\  \mbox{ for each primitive $p$-th root of unity $\xi$,}
\end{equation}
where smooth means that the reduction modulo $p$ of the equations  (\ref{eq:grad-fd}) have $0$ as only solution over an algebraic closure of $\FF_p$. 
If $s=0$  then the reduction of $f_d$ modulo $p$ is smooth whenever $p$ is large and thus Conjecture \ref{con1} for squarefree $N$ and with $\varepsilon=0$ follows for $f$ with $s=0$ (note the different exponent of $p$ in (\ref{eq:deligne}) and of $N$ in (\ref{con1} when $d>2$).    
We proceed by induction on $s$ by restricting $f$ to hyperplanes, as follows. The bound  (\ref{eq:deligne}) for all large $p$ is our base case when $s=0$. Now suppose that $s>0$. After a linear coordinate change of $\AA_\ZZ^n$, we may suppose that the polynomial $g(\hat x) := f(0,\hat x)$ in the variables $\hat x = (x_2,\ldots,x_n)$ is still of degree $d$ and that its degree $d$ homogeneous part $g_d$ has critical locus of dimension $s-1$.  
Hence, for large prime $p$, the reduction of $g_d$ modulo $p$ has also critical locus of dimension $s-1$, in $\AA_{\FF_p}^{n-1}$.
Hence, for large $p$, one has by induction on $s$ that
\begin{equation}\label{eq:deligne-ga}
|\frac{1}{p^{n-1} } \sum_{\hat x=(x_2,\ldots,x_n)\in \FF_p^{n-1}}  \xi^{f(a,\hat x)} | \leq (d-1)^{n-s} p^{-(n-s)/2}
\end{equation}
for each $a\in \FF_p$ and each primitive $p$-th root of unity $\xi$.
Indeed, the polynomial $f(a,\hat x)$ has $g_d$ mod $p$ as its degree $d$ homogeneous part for each $a\in \FF_d$. Now, summing over $a\in \FF_p$ and dividing by $p$  gives
\begin{equation}\label{eq:deligne:s}
|\frac{1}{p^{n} } \sum_{x\in \FF_p^{n}}  \xi^{f(x)} | \leq (d-1)^{n-s} p^{-(n-s)/2}
\end{equation}
for large $p$ (coming from the condition that the reduction of $g_d$ modulo $p$ has critical locus of dimension $s-1$).
Conjecture \ref{con1} with $\varepsilon=0$ for squarefree integers $N$ thus follows, by comparing the exponents in the upper bounds of (\ref{eq:deligne:s}) and (\ref{bound:SfN}), which allows to swallow the constant $(d-1)^{n-s}$ when $d>2$.
(Alternatively, one can use the much more general Theorem 5 of \cite{Katz} when $d>2$ and an argument as in Section \ref{sec:ev2.bis} when $d=2$.)

\subsection{}\label{sec:HeathB}
When one restricts to integers $N$ which are cube free (namely, not divisible by a nontrivial cube), then Conjecture \ref{con:SfN} with $\varepsilon=0$ follows exactly in the same way as for squarefree $N$, but now using both the bounds from \cite{Heath-B-cube-free} and from \cite{DeligneWI}.
Indeed, this similarly gives
\begin{equation}\label{eq:cube-free:s}
 E_f(p^2,\xi)  \leq (d-1)^{n-s} p^{-(n-s)}
\end{equation}
for large $p$ and all $\xi$. Together with the squarefree case, this implies the cube free case of Conjecture \ref{con1}, with $\varepsilon=0$ (note again the different exponent of $p^2$ in (\ref{eq:cube-free:s}) and of $N$ in (\ref{con1}), when $d>2$). 
In fact, with some more work we can go up to $d+2$-th powers instead of just cubes, as follows.
\begin{prop}\label{lem:d+2}
Conjecture \ref{con1} with $\varepsilon=0$
holds when restricted to integers $N$ which are not divisible by a non-trivial $d+2$-th power.
In detail, let $f$, $n$, $s,$ and $d$ be as in the introduction. 
Then there is a constant $c=c(f_d)$ (depending only on $f_d$)
such that for all integers $N>0$ which are not divisible by a non-trivial $d+2$-th power and all primitive $N$-th roots $\xi$ of $1$, one has
\begin{equation}\label{bound:SfNd+2}
E_f(N,\xi) \le c N^{-\frac{n-s}{d}}.
\end{equation}
\end{prop}

We will prove Proposition \ref{lem:d+2} by making a link between $E_f(p^m,\xi)$  and finite field exponential sums, as follows.
For any prime $p$, any $m>0$, any point $P$ in $\FF_p^n$ and any $\xi$, write
\begin{equation}\label{EPf}
E^{P}_f(p^m,\xi)  :=\abs{ \frac{1}{p^{mn}}  \sum_{x\in P + (p\ZZ/p^m\ZZ)^n}  \xi^{f(x)} }.
\end{equation}
Compared to $E_f(p^m,\xi)$, the summation set for $E^{P}_f(p^m,\xi)$ has $p$-adically zoomed in around the point $P$.

Let us consider the following positive characteristic analogues.
\begin{equation}\label{eq:Eft}
E_f(t^m,\psi) := \abs{ \frac{1}{p^{mn} }  \sum_{x\in (\FF_p[t]/(t^m) )^n}  \psi({f(x))} },
\end{equation}
and
\begin{equation}\label{eq:EftP}
E^P_f(t^m,\psi):= \abs{ \frac{1}{p^{mn} }  \sum_{x\in P + (t\FF_p[t]/(t^m) )^n}  \psi({f(x))} },
\end{equation}
for any primitive additive character
$$
\psi:\FF_p[t]/(t^m)\to \CC^\times,
$$
where primitive means that $\psi$ does not factor through the projection
$$
\FF_p[t]/(t^m)\to \FF_p[t]/(t^{m-1}).
$$

The sums of (\ref{eq:Eft}), resp.~(\ref{eq:EftP}), can be rewritten as finite field exponential sums, to which classical bounds like (\ref{eq:deligne:s}) apply.  This is done 
by identifying the summation set 
with $\FF_p^{mn}$, resp.~
with $\FF_p^{(m-1)n}$, namely by sending a polynomial in $t$ to its coefficients, while forgetting the constant terms in the second case. 

We first prove the following variant of Proposition \ref{lem:d+2}.
\begin{prop}\label{lem:d+2:t}
Let $f$, $n$, $s,$ and $d$ be as in the introduction. Then there is a constant $M$ (depending only on $f_d$)
such that for all primes $p$ with $p>M$, all integers  $m>0$ with $m \le d+1$, and all primitive additive characters $\psi:\FF_p[t]/(t^m)\to \CC^\times$
one has
\begin{equation}\label{bound:SfNd+2:t}
E_f(t^m,\psi) \le p^{-m\cdot \frac{n-s}{d}}.
\end{equation}
\end{prop}
\begin{proof}
By a reasoning as for the squarefree case, it is sufficient to treat the case with $s=0$ for large $p$, while letting the $f_i$ for $i<d$ vary over homogeneous polynomials in $\FF_p[t,x]$ of degree $i$ in $x$, and while keeping $f_d$ fixed in $\ZZ[x]$.
So, we may assume that $s=0$, and, by the squarefree case treated above, that $m>1$. We also may assume that $d\ge 3$ by the above treatment of the case $d=2$.
%
For each $p$, let $C_p$ be the set of critical points of the reduction of $f$ modulo $p$. 
Since $s=0$, one has $\# C_p \le c_1$ for some constant $c_1$ depending only on $n$ and $d$, see for example the final inequality of \cite{Heath-B-cube-free}, or Lemma \ref{criticalpoint} below.
Clearly we have
\begin{equation}\label{eq:count0}
E_f(t^m,\psi)  = \sum_{P\in C_p}   E^{P}_f(t^m,\psi)
\end{equation}
for all primes $p>d$, all $m>1$ and all primitive $\psi:\FF_p[t]/(t^m)\to \CC^\times$.
For $m< d$,  note that 
\begin{equation}\label{eq:countP}
\frac{1}{p^{mn}} \cdot \# (t\FF\llb t\rrb /(t^m))^n =  p^{ n(m-1)   -mn} = p^{-n} < p^{-mn/d}.
\end{equation}
For $m< d$,  we thus find by (\ref{eq:count0}) that
\begin{equation}\label{eq:countPP}
E_f(t^m,\psi)\le c_1 p^{-n},
\end{equation}
and  (\ref{bound:SfNd+2:t}) follows when $m<d$ and $p$ is large enough 
so that the constant $c_1$ is eaten by the extra saved power of $p$ coming from (\ref{eq:countP}).
We now treat the case that $m=d$. 
If $p>d$ is such that the reduction of $f$ modulo $t$ is smooth homogeneous of degree $d$, then $C_p=\{0\}$ and there is nothing left to prove since then $E_f(t^m,\psi)=E^{P_0}_f(t^m,\psi) \le p^{-n}=p^{-mn/d}$, with $P_0=\{0\}$.
If the reduction of $f$ modulo $t$ is not homogeneous of degree $d$, and, $p$ is larger than $d$, then there is a constant $c_2$ (depending only on $n$ and $d$) such that
\begin{equation}\label{eq:countPPP}
E^{P}_f(t^m,\psi) \le c_2 p^{-n - 1/2}
\end{equation}
for all $P$ in $C_p$ and all primitive $\psi$. Indeed, this follows from the worst case of (\ref{eq:deligne-ga}) applied to $E^P_f(t^m,\psi)$, after rewriting it as a finite field exponential sum as explained just above the proposition. The case $m=d$ for  (\ref{bound:SfNd+2:t}) follows, where the constant $c_2$ is eaten by the extra saved power of $p$ coming from $d\ge 3$ and (\ref{eq:countPPP}). For $d=m+1$, when we rewrite
$E^P_f(t^m,\psi)$ for $P\in C_p$ as a finite field exponential sum over $(m-1)n$ variables running over $\FF_p$, we can again apply (\ref{eq:deligne-ga}), now in $n(m-1)$ variables and with highest degree part $f_d$ which has singular locus of dimension $n(m-2)$ inside $\AA^{n(m-1)}$.  Since $d\ge 3$, we again can use the extra saved power of $p$ to obtain (\ref{bound:SfNd+2:t}) and the proposition is proved.
\end{proof}

The transfer principle for bounds from  Theorem 3.1 of \cite{CGH4} can be applied to compare the exponential sums $E_f(p^m,\xi)$ and $E_f(t^m,\psi)$; we will use it in the following basic form.  Recall that a Presburger subset $A$ of $\NN$ is a Boolean combination of congruence classes and subintervals of $\NN$.
\begin{cor}[\cite{CGH4}]\label{cor:CGH4}
Let $g$ a homogeneous polynomial over $\ZZ$, of degree $d>1$ and in $n$ variables. Consider a real number $\sigma>0$ and a Presburger subset $A\subset \NN$. Then the following two statements are equivalent.
\begin{enumerate}
\item There exist constants $c$ and $M$ such that for all primes $p>M$ and all polynomials $f$ in $\ZZ_p[x_1,\ldots,x_n]$  of degree $d$ and with  $f_d=g$, one has
\begin{equation*}\label{transfer.bound}
E_f(p^m,\xi) \leq c p ^{-\sigma m}
\end{equation*}
for all $m\in A$  and all primitive $p^m$-th roots of unity $\xi$.

\item There exist constants $c'$ and $M'$ such that for all primes $p>M'$ and all polynomials $f$ in $\FF_p[[t]][x_1,\ldots,x_n]$ of degree $d$ and such that  $f_d  = g \bmod (p)$ holds in $\FF_p[x]$,  one has
\begin{equation*}\label{transfer.bound.t}
E_f(t^m,\psi) \leq c' p ^{-\sigma m}
\end{equation*}
for all $m\in A$ and all primitive 
characters $\psi:\FF_p[t]/(t^m)\to \CC^\times$.
\end{enumerate}
\end{cor}
In the corollary, we have extended the notation $E_f$ to more general $f$, namely with more general coefficients than just in $\ZZ$, in the obvious way.
More generally than Corollary \ref{cor:CGH4}, Theorem 3.1 of \cite{CGH4} allows to transfer bounds that hold for motivic families of functions, and, the families in the corollary are a special case of such family.
\begin{proof}[Proof of Corollary \ref{cor:CGH4}]
Clearly the left hand sides and of the right hand sides of the inequalities come from motivic functions $H$ and $G$ as required in Theorem 3.1 of \cite{CGH4}. Now the corollary readily follows from the conclusion of Theorem 3.1 of \cite{CGH4} for such $H$ and $G$.
\end{proof}

\begin{proof}[Proof of Proposition \ref{lem:d+2}]
We show
for all large primes $p$, all integers  $m>0$ with $m \le d+1$, and all primitive $p^m$-th roots of $1$,  that
\begin{equation}\label{bound:SfNd+2:pp}
E_f(p^m,\xi) \le p^{-m\cdot \frac{n-s}{d}},
\end{equation}
where moreover the lower bound on $p$ depends only on $f_d$. The case $d=2$ is already shown.
For $m\not=d\ge 3$ this follows at once from Corollary \ref{cor:CGH4} and the corresponding extra power savings when $m\not=d$ in the proof of Proposition \ref{lem:d+2:t}.
Indeed, the transfer principle holds uniformly in $f$ as long as $f_d$ and $n$ are fixed, since these bounds (with the extra power savings) from the proof of Proposition \ref{lem:d+2:t} depend only on $f_d$ and $n$.
Let us now treat the remaining case of $m=d$. It is again enough to treat the case $s=0$. For a prime $p>d$ such that the reduction of $f$ modulo $p$ is not smooth homogeneous of degree $d$,  we are done similarly by the transfer principle for bounds from \cite{CGH4} and the corresponding power savings in the proof Proposition \ref{lem:d+2:t}.
If $m=d$ and $p>d$ is such that the reduction of $f$ modulo $p$ is smooth homogeneous of degree $d$, then we have that $P_0=\{0\}$ is the unique critical point of the reduction of $f$ modulo $p$, and thus
$$
E_f(p^m,\xi)=E^{P_0}_f(p^m,\xi) \le p^{-n}=p^{-mn/d}.
$$
The proof of Proposition \ref{lem:d+2} is thus finished.
\end{proof}

\section{The non-degenerate case}\label{sec:CAN}

In this section we show that Conjecture \ref{con:SfN} with $\varepsilon=0$ holds for non-degenerate polynomials, where the non-degeneracy condition is with respect to the Newton polyhedron of $f-f(0)$ at zero as in \cite{CAN} (which is slightly different than the non-degeneracy notion of \cite{Kuchn}, \cite{Varchen}). The non-degeneracy condition generalizes the situation where $f$ is a sum of monomials in separate variables, like $x_1x_2+x_3x_4$.
In detail, writing $f(x)-f(0) = \sum_{i\in\NN^n} \beta_i x^i$ in multi-index notation, let
$$
\Supp_f  :=\{i\in\NN^n\mid \beta_i\not=0\}
$$
be the support of $f-f(0)$. Further, let
$$
\Delta_0(f) := \Conv (\Supp_f + (\RR_{\geq 0})^n)
$$
be the convex hull of the sum-set of $\Supp_f$ with $(\RR_{\geq 0})^n$ where $\RR_{\geq 0}$ is $\{x\in\RR\mid x\ge 0\}$. The set $\Delta_0(f)$ is called the Newton polyhedron of $f-f(0)$ at zero.
Let $\sigma_f$ be the unique real value such that $(1/\sigma_f,\ldots,1/\sigma_f)$ is contained in a proper face of $\Delta_0(f)$.
Further, let $\kappa$ denote the maximal codimension in $\RR^n$ of $\tau$ when $\tau$ varies over the faces of $\Delta_0(f)$  containing $(1/\sigma_f,\ldots,1/\sigma_f)$.
For each face $\tau$ of the polyhedron $\Delta_0(f)$, consider the polynomial
$$
f_\tau := \sum_{i\in\tau} \beta_i x^i.
$$
Call $f$ non-degenerate with respect to $\Delta_0(f)$ when for each face $\tau$ of $\Delta_0(f)$ and each critical point $P$ of $f_\tau:\CC^n\to\CC$, at least one coordinate of $P$ equals zero. Recall that a complex point $P\in\CC^n$ is called a critical point of $f_\tau$ if $\partial f_\tau /\partial x_i (P)=0$ for all $i=1,\ldots,n$.

The following proposition  slightly extends the main result of \cite{CAN} as it covers small primes as well.  Note that \cite[Theorem 1.4.1]{CAN} gives evidence for Igusa's conjecture on exponential sums in the variant of \cite[Conjecture 1.2]{CVeys}.

\begin{prop}[{\cite[Theorem 1.4.1]{CAN}}]\label{prop:CANsigma}
Suppose that $f$ is non-degenerate with respect to~$\Delta_0(f)$.
Then, there is a constant $c$ such that for all primes $p$, all integers $m\ge 2$ and all primitive $p^m$-th roots of unity $\xi$ one has
\begin{equation}\label{eq:CAN}
E_f(p^m,\xi) \le c p^{-m\sigma_f}m^{\kappa-1}.
\end{equation}
\end{prop}

From Proposition \ref{prop:CANsigma}  we will derive the following evidence for Conjecture \ref{con1}. 
\begin{thm}
\label{prop:CAN}
Let $f$, $n$, $s$, and $d$ be as in the introduction. Suppose that $f$ is non-degenerate with respect to $\Delta_0(f)$.
Then Conjecture \ref{con1} with $\varepsilon=0$ holds for $f$. Namely,
there is $c$ such that for all integers $N>0$ and all primitive $N$-th roots of unity $\xi$, one has
$$E_f(N,\xi) \le c N^{-\frac{n-s}{d}}.$$

Furthermore, for all large  primes $p$ (with `large' depending on $f$), all integers $m>0$ and all primitive $p^m$-th roots of unity $\xi$ one has
$$E_f(p^m,\xi) \le p^{-m\frac{n-s}{d}}.$$
\end{thm}


\begin{proof}[Proof of Proposition \ref{prop:CANsigma}]
In  \cite{CAN}  it is shown that one can take a constant $c$ so that (\ref{eq:CAN}) holds for all large primes $p$ and al $m\ge 2$. So, there is only left to prove that for each prime $p$ there is
a constant $c_p$ (depending on $p$) such that for each integer $m\ge 2$ one has
\begin{equation}\label{eq:CANp}
E_f(p^m,\xi) \le c_p p^{-m\sigma_f}m^{\kappa-1}.
\end{equation}
Indeed, (\ref{eq:CANp})  is only used for the finitely many remaining primes.
First, if $f$ is non-degenerate w.r.t~$\Delta_0(f)$ we show that $f(0)$ is the only possible critical value of $f$, by induction on $n$. If $n=1$, by the non-degeneracy of $f$, we get that $f$ has no critical point in $\CC^\times$ and we are done.  Now suppose that  $n>1$. Let $f$ be a polynomial in $n$ variables which is  non-degenerate w.r.t~$\Delta_0(f)$. Suppose that $u=(u_1,...,u_n)$ is a critical point of $f$. By the non-degeneracy of $f$ there exists $1\leq j\leq n$ such that $u_j=0$. Without lost of generality we can suppose that $j=n$. We can write $f=f(0)+ \sum_{i=0}^d g_i(x_1,...,x_{n-1})x_n^i$ for some polynomials $g_i$ with furthermore $g_0(0)=0$.  It is sufficient to show that $f(u_1,...,u_n)=f(0)$. Since $u_n=0$, it suffices to show that $g_0(u_1,...,u_{n-1})=0$. By the fact that $f$ is non-degenerate w.r.t~$\Delta_0(f)$ we get that $g_0$ is non-degenerate w.r.t~$\Delta_0(g_0)$. It is clear that $(u_1,...,u_{n-1})$ is a critical point of $g_0$. So, we can use the inductive hypothesis to deduce that $g_0(u_1,...,u_{n-1})=g_0(0)=0$.
Now, since $f$ has no other possible critical value than $f(0)$ and since there exists a toric log resolution of $f-f(0)$ whose numerical properties (in particular its discrepancy numbers) are controlled by $\Delta_0(f)$ (see for example \cite{Varchen}), inequality (\ref{eq:CANp}) follows from Igusa's work in \cite{Igusa3}. Here, we use the following information on the discrepancy numbers coming from the toric log resolution $\pi$ of $f-f(0)$, in relation to $\Delta_0(f)$. If $E$ is an irreducible component of the exceptional locus of $\pi$ and if one writes $N_E$ for the multiplicity of $E$ in the divisor $\pi^*(f-f(0))$  and $\nu_E-1$ for the multiplicity of $E$ in  $\pi^*(dx_1\wedge\ldots\wedge dx_n)$, then one has $\nu_E/N_E\ge \sigma$. Furthermore, any intersection of $\kappa+1$ many such $E$ for which the equality $\nu_E/N_E = \sigma$ holds is empty. Since $f(0)$ is the only critical value of $f$, we are now done by Igusa's work in \cite{Igusa3}.
\end{proof}

The proof of Theorem \ref{prop:CAN} relies on Proposition \ref{prop:CANsigma} and Lemma \ref{lem:CAN}. Note that the following lemmas \ref{lem:CAN} and \ref{s=0} do not require $f$ to be non-degenerate.

\begin{lem}\label{lem:CAN} Let $f$, $n$, $s$, and $d$ be as in the introduction. (In particular, $f$ is allowed to be inhomogeneous and there is no condition on non-degeneracy.) Suppose that $d\geq 3$. Then one has $\sigma_f \ge (n-s)/d$, and, equality holds if and only if there is a smooth form $g$ of degree $d$ in $n-s$ variables  
such that
$$
f(x) - f(0)
=g(x_{i_1},\ldots,x_{i_{n-s}})
$$
for some $i_j$ with $1\leq i_1<i_2<\ldots<i_{n-s}\leq n$.
\end{lem}


We will first prove Lemma \ref{lem:CAN} in the case that $s=0$, 
using the following lemma. We write $\Conv(\Supp_f)$ for the convex hull of $\Supp_f$ in $\RR^n$.
\begin{lem}\label{s=0} Let $f$, $n$, $s$, and $d$ be as in the introduction. Suppose furthermore that $d\geq 3$, $s=0$ and that $f=f_d$, namely, $f$ is smooth homogeneous of degree at least $3$. Then
$$
\dim(\Conv(\Supp_f))=n-1,
$$
and, the point $(d/n,\ldots,d/n)$ belongs to the interior of $\Conv(\Supp_f)$. In particular, $\sigma_f=n/d$ and $\kappa=1$.
\end{lem}
\begin{proof}Since $f=f_d$, it is clear that $\dim(\Conv(\Supp_f))\leq n-1$. Suppose now that either $\dim(\Conv(\Supp_f))<n-1$, or, that $(d/n,\ldots,d/n)$ does not belong to the interior of $\Conv(\Supp_f)$. We try to find a contradiction. By our assumptions, there exists a hyperplane $H=\{a\in \RR^n \mid k\cdot a=0\}$ for some $k\in \RR^n\setminus\{0\}$ such that the point 
$(d/n,\ldots,d/n)$ belongs to $H$ and such that $\Supp_f$ belongs to the half space $H_+:= \{a\in \RR^n \mid k\cdot a\ge 0\}$. Let $I$ be the subset of $\{1,\ldots,n\}$  consisting of $i$ with  $k_i>0$ and let $J$ be the subset of $\{1,\ldots,n\}$  consisting of $j$ with  $k_j<0$. Clearly $I$ and $J$ are disjoint. Since $(d/n,\ldots,d/n)$ belongs to $H$, it follows that  $I$ and $J$ are both nonempty and that
\begin{equation}\label{=}
\sum_{i\in I} k_i=\sum_{j\in J} |k_j|.
\end{equation}
Furthermore, the inclusion $\Supp_f\subset H_+$ implies that
\begin{equation}\label{*}
\sum_{i\in I} k_ia_i\geq \sum_{j\in J} |k_j|a_j  \mbox{ for all $a\in \Supp_f$.  }
\end{equation}

%
%

Consider the set
$
\Supp_f^0
$
 consisting of those $a\in \Supp_f$ with moreover $a_i=1$ for some $i\in I$ and $a_{i'}=0$ for all $i'\not\in J\cup \{i\}$.
For $a\in \Supp_f^0$ write $t(a)$ for the unique $i\in I$ with $a_i=1$ and write
$$
I_0:=\{i\in I\mid  \exists a\in \Supp_f^0 \mbox{ with } t(a)=i 
\}.
$$
Clearly we can write
$$
f(x)=\sum_{i\in I_0}x_ig_i(x_j)_{j\in J}+\sum_{a\in \Supp_f\setminus \Supp_f^0
}\beta_{a}x^a
$$
for some polynomials $g_i$ in the variables $(x_j)_{j\in J}$.
Also, the algebraic set
$$
\bigcap_{i\notin J}\{x_i=0 \} \bigcap_{i\in I_0}\{g_i=0\}
$$
in $\AA^n_\CC$ has dimension at least $|J|-|I_0|$  and is contained in $\Crit_f$, the critical locus of $f$. By our smoothness condition $s=0$, this implies
\begin{equation}\label{I0J}
|I_0|\geq |J|.
\end{equation}
Hence, we can write $I_0=\{i_1,\ldots,i_\ell\}$ and $J=\{j_1,\ldots,j_m\}$ with $m\leq \ell$ and with
\begin{equation}\label{kikj}
k_{i_1}\geq k_{i_2}\geq\ldots\geq k_{i_{\ell}} \mbox{ and } |k_{j_1}|\geq |k_{j_2}|\geq\ldots\geq |k_{j_{m}}|.
\end{equation}
To prove the lemma it is now sufficient to show that
\begin{equation}\label{kirjr}
k_{i_{r}}>|k_{j_r}| \mbox{ for all $r$ with } 1\leq r\leq m.
\end{equation}
Indeed,  (\ref{kirjr}) gives a contradiction with (\ref{=}).
To prove (\ref{kirjr}), we suppose that there is  $r_0$ with $1\leq r_0\leq m$ and with
\begin{equation}\label{kr0r1}
k_{i_{r_0}}\le |k_{j_{r_0}}|
\end{equation}
and we need to find a contradiction.
If there exists $a\in \Supp_f^0$ such that $a_{j_{r_1}}\ge 1$ for some $r_1\le r_0$, then let $a$ be such an element and let $t$ be $t(a)$; otherwise,  let $a$ be arbitrary and put $t=0$.
We will now show that $t<r_0$. If $t=0$ this is clear, so, suppose that $t>0$.
Since $d\ge 3$ and $a\in \Supp_f^0$, 
we find by (\ref{*})  and (\ref{kikj}) that 
\begin{equation}\label{kit}
k_{i_{t}} =   \sum_{i\in I} k_ia_i \geq \sum_{j\in J} a_{j}|k_{j}|  >  |k_{j_{r_1}}|\ge |k_{j_{r_0}}|.
\end{equation}
Together with (\ref{kikj}) and (\ref{kr0r1}), 
this implies that $t< r_0$ as desired. 
We can thus write
\begin{equation}\label{fr}
f=\sum_{1\leq \ell \le r_0-1}x_{i_\ell}h_\ell(x_{j})_{j\in J}+\sum_{a\in A
}\beta_ax^a
\end{equation}
with
$$
A=\{a\in \Supp_f\mid  \sum_{i\notin\{j_1,\ldots,j_{r_0}\}}a_i\geq 2\}
$$
and with some polynomials $h_\ell$ in the variables $(x_j)_{j\in J}$.
It follows that the algebraic set
$$
\bigcap_{i\notin \{j_1,\ldots,j_{r_0}\}}\{x_i=0\}\bigcap_{1\leq \ell\leq r_0-1}\{h_{\ell}=0\}
$$ 
has dimension at least $1$ and is contained in $\Crit_f$, again a contradiction with our smoothness assumption $s=0$. So, relation (\ref{kirjr}) follows and the lemma is proved.
\end{proof}

The case of Lemma \ref{lem:CAN} with $s=0$ is derived from Lemma \ref{s=0}, as follows.
\begin{proof}[Proof of Lemma \ref{lem:CAN} with $s=0$]
Let $f$ be of degree $d\geq 3$ and with $s=0$. We need to show that $\sigma_f \ge n/d$, and, that $\sigma_f = n/d$ if and only if $f=f_d$.
Since $f_d$ is smooth, Lemma \ref{s=0} implies that $(d/n,\ldots,d/n)$ belongs to $\Delta_0(f)$, 
and hence,  $\sigma_f\geq n/d$. Suppose now that 
$f\neq f_d$. Then there exists $a\in \Supp_f$ with $
\sum_{i=1}^n a_i<d$.  Hence, by Lemma \ref{s=0} and the definition of $\Delta_0(f)$,  there exists $\varepsilon>0$ such that
$$
\{x\in\RR^n\mid ||x-(d/n,\ldots,d/n)||\leq \epsilon 
\} \subset \Delta_0(f).
$$
Therefore it is clear that $\sigma_f>n/d$. 
This finishes the proof of Lemma \ref{lem:CAN} with $s=0$.
\end{proof}
\begin{proof}[Proof of lemma \ref{lem:CAN} with $s>0$] To prove the lemma with $s>0$ we may suppose that 
\begin{equation}\label{eq:sf}
\sigma_f\leq (n-s)/d.
\end{equation}
 By the definition of $\sigma_f$ we have
\begin{equation}\label{eq:sf2}
\min_{a\in \Conv(\Supp_f)}\max(a)=1/\sigma_f, 
\end{equation}
where $\max(a)=\max_{1\leq i\leq n}\{a_i\}$ and where $\Conv(\Supp_f)$ is the convex hull of $\Supp_f$.
We set
$$
k:=\min_{\max(a)=1/\sigma_f} \# \{i|a_i=1/\sigma_f\},
$$
where the minimum is taken over $a\in \Conv(\Supp_f)$.
Let $a\in \Conv(\Supp_f)$ realize this minimum, namely, with  $\# \{i\mid a_i=1/\sigma\}=k$ and with $\max(a)=1/\sigma_f$.
We may suppose that
\begin{equation*}\label{eq:maxa}
a_1=\ldots =a_k=1/\sigma_f \mbox{ and $a_i<1/\sigma_f$ if } i>k.
\end{equation*}
Let $b\in \Conv(\Supp_f)$ be such that $\max(b)=1/\sigma_f$.  Then,  for each $\lambda\in[0,1]$, the point $c_\lambda:=\lambda a+(1-\lambda)b$ lies in $\Conv(\Supp_f)$. When $\lambda$ is sufficiently close to $1$, then we have $c_{\lambda,i}<1/\sigma_f$ for all $i>k$, and, the definition of $k$ implies that $b_i=1/\sigma_f$ for all $1\leq i\leq k$. By the same reasoning, for each  $b\in \Conv(\Supp_f)$ one has $b_i\geq 1/\sigma_f$ for some $i$ with $1\leq i\leq k$. The definition of $k$ and (\ref{eq:sf2}) also tell us that $k/\sigma_f\leq d$, and thus we find
\begin{equation}\label{eq:sf3}
k\leq n-s
\end{equation}
from  (\ref{eq:sf}). For any tuple of complex numbers $C=(c_{i,j})_{1\leq i,j\leq s} 
$
we consider the polynomial
$$
g_C=f(x_1,\ldots,x_{n-s}, x_{n-s+1}+\sum_{1\leq j\leq n-s}c_{1,j}x_j,\ldots,x_n+\sum_{1\leq j\leq n-s}c_{s,j}x_j).
$$
For a generic choice of $C$ one has $\Supp_f\subset \Supp_{g_C}$. Furthermore, we show that for a generic choice of $C$ the polynomial
$$
h_C=f_d(x_1,\ldots,x_{n-s}, \sum_{1\leq j\leq n-s}c_{1,j}x_j,\ldots,\sum_{1\leq j\leq n-s}c_{s,j}x_j)
$$
is smooth homogeneous in $n-s$ variables, where $f_d$ is the degree $d$ homogeneous part of $f$.
For a generic choice of $e_{n}=(e_{n,i})_{i < n}$ one has
$$
\dim(\Sing(f_{d,e_{n}}))=n-s,
$$
where
$$
f_{d,e_{n}}(x_1,\ldots,x_{n-1}) :=f_{d}(x_1,\ldots,x_{n-1}, \sum_{i=1}^{n-1}e_{n,i}x_i),
$$
considered as a polynomial in $n-1$ variables $x_i$ with $i<n$.
We repeat this argument to see that for a generic choice  $E=(e_{n-s+1},\ldots,e_{n})$ with $e_{j}=(e_{j,i})_{i<j}$ one has that
$$
\dim(\Sing(f_d|_{V_E}))=n-s,
$$
where
$$
V_E=\{x|x_{n}=\sum_{i<n}e_{n,i}x_i, \ldots, x_{n-s+1}=\sum_{i\leq n-s}e_{n-s+1,i}x_i\}.
$$
It is clear that the smoothness of $f_d|_{V_E}$ for generic $E$ corresponds to the smoothness of $h_C$ for generic $C$.
Let us fix such a choice of $C$ with all these properties, namely, that $h_C$ is smooth and that $\Supp_f\subset \Supp_{g_C}$. If $a\in \Supp_{g_C}$, it is easy to see that
$a_i\geq b_i$ for all $i$ with $1\leq i\leq n-s$ and for some $b\in \Supp_f$. Hence, $\sigma_{g_C}\leq \sigma_f$, by the definition of $k$ and our chosen ordering of the coordinates.
On the other hand, from $\Supp_f\subset \Supp_{g_C}$ it follows that
$\sigma_{g_C}\geq \sigma_f$, and hence, we have
$$
\sigma_{g_C}=\sigma_f.
$$
Let $\pi$ be the coordinate projection from $\RR^n$ to $\RR^{n-s}$. Then, for any $e=(e_j)_{j=1,\ldots,s}$, consider the polynomial
\begin{align*}
g_{C,e}(x_1,\ldots,x_{n-s}) := g_C(x_1,\ldots,x_{n-s}, e_1,\ldots,e_s).
\end{align*}
Then, for generic choice of $e$, we have
$$
\Supp_{g_{C,e}} =  \pi(\Supp_{g_C}).
$$
Let us fix such a choice of $e$. It is clear that
$$
\sigma_{g_{C,e}}=\sigma_{g_C},
$$
by the definition of $k$ and our ordering of the coordinates. Note that the highest degree homogeneous part of $g_{C,e}$ equals $h_C$, which is smooth. Thus, we can use Lemma \ref{lem:CAN} with $s=0$ (which is already proved) for  $g_{C,e}$. So, we find
$$
\sigma_{g_{C,e}}=(n-s)/d, \mbox{ and, } g_{C,e}-g_{C,e}(0)=h_C.
$$
Hence,
$$
\pi(\Supp_f)\subset\pi(\Supp_{g_C})\subset\{a\in \RR^{n-s}|a_1+\ldots+a_{n-s}=d\}.
$$
This holds if and only if $f-f(0)=f_d=h(x_1,\ldots,x_{n-s})$ for some polynomial $h$, which is smooth homogeneous since $\dim(\Crit_{f_d})=s$. This finishes the proof of the lemma \ref{lem:CAN}.
\end{proof}

We are now ready to prove Theorem \ref{prop:CAN}.

\begin{proof}[Proof of Theorem \ref{prop:CAN}]
The case that $d=2$ is treated in Section \ref{sec:ev2.bis}.
Hence, we may suppose that $d\geq 3$. By Proposition \ref{prop:CANsigma}, there exists a constant $c_2$ such that for all integers $m>1$, all primes $p$ and all primitive $p^m$-th roots of unity $\xi$ we have
\begin{equation}\label{m>1}
E_f(p^m,\xi) \le c_2p^{-m\sigma_f}m^{\kappa-1}.
\end{equation}
By Lemma \ref{lem:CAN} we have $\sigma_f\geq \frac{n-s}{d}$. If $\sigma_f>\frac{n-s}{d}$, then we have $\frac{n-s}{d}<\frac{n-s}{2}$, from using $d\ge 3$ and $s<n$.
Conjecture \ref{con:SfN} for this case follows by combining (\ref{eq:p-i}) and (\ref{m>1}) with the squarefree case from Section \ref{sec:ev3} .
If $\sigma_f=\frac{n-s}{d}$, we use Lemma \ref{lem:CAN} again  to see that $f=g_d+f(0)$ for a smooth form $g_d$ of degree $d$ in $n-s$ variables. Conjecture \ref{con:SfN} for this case follows by Igusa's case from Section \ref{sec:ev1}.
\end{proof}



\begin{rem}\label{rem:gen}
If $f$ is weighted homogeneous, then the notion of non-degenericity with respect to $\Delta_0(f)$ is generic, but otherwise the genericity is more subtle, by the difference between `critical points' and `singular points'.
In fact, whether or not the notion of non-degenericity with respect to $\Delta_0(f)$ is generic depends on $\Supp_f$.
When  $\Supp_f$ is contained in a hyperplane which does not contain the origin $0$ and has a normal vector with non-negative coordinates (see \cite[Section 2.2]{CAN}), then the condition of non-degeneracy on the coefficients $\beta_i$ is generic within this support, that is, for any $\gamma$ outside a Zariski closed subset of $\CC^{\Supp_f}$, the polynomial $\sum_{i\in \Supp_f} \gamma_i x^i$ is non-degenerate with respect to its Newton polyhedron at zero. This hyperplane condition generalizes the case of weighted homogeneous polynomials. However, in the general case, this genericity may be lost since we imposed conditions on critical points of $f_\tau$ instead of on singular points as is done more traditionally in \cite{Kuchn}, \cite{Varchen}. Especially for $\tau = \Delta_0(f)$ this makes a difference when the mentioned hyperplane condition is not met. For instance, polynomials of the form $f(x)=ax^3+by^3+cxy$ for nonzero $a$, $b$, and $c$ are never non-degenerate in our sense, the problem being with $\tau = \Delta_0(f)$.
\end{rem}

\section{The four variable case}\label{sec:n=4}
In this final section we prove Conjecture $1$ when $n\le 4$ (Theorem \ref{allp}), and a slightly stronger result when furthermore $d=3$ and $s=0$ (Proposition \ref{prop:3}).  
\begin{thm}\label{allp}
Let $f$, $n$, $s,$ and $d$ be as in the introduction and suppose that  
$n\le 4$. Then Conjecture \ref{con1} holds for $f$.
\end{thm}

The proof of Theorem \ref{allp} relies on Weierstrass preparation, properties of $\hat\alpha_f$ based on results on minimal exponents from \cite{MustPopa}, bounds from \cite{CMN}, and Igusa's results as exposed in \cite{DenefBour}.
%
The following auxiliary lemma is well known, see for example the final inequality of \cite{Heath-B-cube-free}, where furthermore an explicit upper bound on the number of critical points is obtained.
\begin{lem}\label{criticalpoint}
Suppose that $g=g_0+...+g_d$ is a polynomial in $\CC[x_1,...,x_n]$ of degree $d$ and with
$\dim(\Crit_{g_d})=0$, where $g_i$ is the degree $i$ homogeneous part of $g$, and where $\Crit_{g_d}$ is the critical locus of $g_d:\CC^n\to\CC$.
Then $\Crit_g$ is a finite set, with $\Crit_g$ the critical locus of $g$.
\end{lem}
\begin{proof}
This is shown by homogenizing $g$ as in the reasoning towards the final inequality of \cite{Heath-B-cube-free}, where it is even shown that
$\# \Crit_g \le (d-1)^n$, by an application of B\'ezout's theorem.
\end{proof}

\begin{proof}[Proof of Theorem \ref{allp}]
Suppose that $n\le 4$.
If $d=2$ or $(n-s)/d\leq 1$, then Conjecture \ref{con:SfN} follows by the arguments in Sections \ref{sec:ev2} and \ref{sec:ev2.bis}.
Hence, we may concentrate on the case that $d=3$ and $s=0$.

Write $\Crit_f$ for the critical locus of $f$ in $\AA^n_\ZZ$.
By Remark \ref{rem:(5)}, Section \ref{sec:ev3} and Lemma \ref{criticalpoint} it is sufficient to show that there exists a constant $c>0$ such that for all primes $p$, all integers $m\ge 2$, all points $P$ in $\Crit_f(\FF_p)$, and all primitive $p^m$-th roots of unity $\xi$ we have
\begin{equation}\label{m>1, p>M}
E^{P}_f(p^m,\xi)\leq c (p^m)^{-n/d +\varepsilon},
\end{equation}
where  $E^{P}_f$ is as in (\ref{EPf}). Indeed, for any point $P$ outside $\Crit_f(\FF_p)$, one has $E^{P}_f(p^m,\xi)=0$ as soon as $p$ is large and $m\ge 2$ (and similarly for small $p$ and $m$ large enough), and, for large $p$ one has $\#\Crit_f(\FF_p)\le \#\Crit_f(\CC)$.

If there exists a point $a\in \Crit_f(\CC)$ such that the multiplicity of $f$ at $a$ is $3$ then such $a$ is unique, and, up to an affine coordinate change over $\CC$ putting $a$ in the origin, one has $f=f_3$.  Hence, in this case that there exists a point $a\in \Crit_f(\CC)$ with  multiplicity $3$, for each $p$ one either has that $a$ belongs to $\ZZ_p^n$, and then one has $E_f(p^m,\xi)=E_{f_3}(p^m,\xi)$ for all $m\geq 1$ and all $\xi$, or, one has that $\Crit_f(\ZZ_p)$ is empty. In the former case one is done by the treatment for $f=f_d$ with $s=0$ of Section \ref{sec:ev1}, and, in the latter case one proceeds as in the above discussion for $P$ outside $\Crit_f(\FF_p)$.
Hence, we may assume that $f$ has multiplicity $2$ at each point in $\Crit_f(\CC)$.

We focus on (\ref{m>1, p>M}) with $P=\{0\}$, and, we assume that $f$ has multiplicity $2$ at zero. (For general  $P$ in $\Crit_f(\FF_p^n)$, one works similarly.)
Up to using a transformation as in Lemma \ref{lem:T}, we may suppose that $f_2$ is diagonal.
We will perform a change of variables and reduce our question to bounding $E^P_F$ instead of $E_f^P$, for some $F$ of the form $x^2+G(y,z,w)$ with $G$ a polynomial, where Weierstrass preparation will be key, and, cutting-off a power series at some high degree $D$.
We first reason for large prime numbers $p$.
Let $M$ be one of the nonzero coefficients of $f_2$, say, the term $Mx^2$ appears in $f_2(x,y,z,w)$. By Weierstrass preparation for $f/M$ in the ring $R[[x,y,z,w]]$ with $R$ being the ring  $\ZZ[1/M]$, we may assume that $f(x,y,z,w)/M$ equals $u(x,y,z,w)(x^2 + xh(y,z,w) +    g(y,z,w))$  for some $g,h$ in $R[[y,z,w]]$ and some unit $u$ in $R[[x,y,z,w]]$ (see e.g.~example 4.4(1) of \cite{CLip} for Weierstrass preparation over 
$R$). Now up to a transformation with new variable $x-h/2$ instead of $x$, we may assume that $h=0$. By the general theory of local zeta functions of  \cite{DenefBour}, and up to changing the primitive root $\xi$, the unit $u$ plays no role since its reduction modulo $p$ is constant (recall that $p$ is assumed to be large). That is, for large $p$ one has that $E_f^P(p^m,\xi) = E_{f/u}^P(p^m,\xi')$ for all large primes, each $m$, each $\xi$, and a corresponding $\xi'$, and with the obvious meaning for $E_{f/u}$.
By Theorem E(3) and Proposition 6.6(1) of \cite{MustPopa}, if $g_D$ is the polynomial which coincides with $g$ up to some large degree $D$, then the minimal exponent at zero of $g$ can differ no more than $3/D$ from the minimal exponent of $g_D$.
Furthermore, as soon as $D$ is large enough, the sums  $E_{g_D}^P(p^m,\xi)$ are independent of the choice of $D$ for each $m\ge 0$ and each $\xi$ and each large prime $p$. This is so because all the data that goes into a computation of $E_{g_D}^P(p^m,\xi)$ via cell decomposition coincides with the corresponding computation of $E^P_{g}(p^m,\xi)$ if $D$ is large enough, yielding $E_{g_D}^P(p^m,\xi)=E^P_{g}(p^m,\xi)$ for all large primes $p$, all $m$, and all $\xi$. Indeed, a cell decomposition as in Section 6 of \cite{CLip} only uses finitely many terms of the occurring analytic functions, and, such cell decomposition together with Proposition (1.4.4)  of \cite{DenefBour} allow one to compute $E_{g_D}^P(p^m,\xi)$ and $E^P_{g}(p^m,\xi)$. Hence, it is sufficient to bound $E_{F}^P$ for $F$ being $x^2 + G(y,z,w)$ where $G$ is the polynomial $g_D$ for some large $D$, say, with $3/D<\varepsilon/3$. We may even suppose that the coefficients of $G$ lie in $\ZZ$, after multiplying with a power of $M$ and changing $\xi$ correspondingly.
Now, if $G$ has multiplicity $3$ or more at $P$, then $\hat \alpha_{G}\le 1$ by Theorem E(3) of \cite{MustPopa}. Thus, we are done by \cite[Theorem 1.5]{CMN} for $G$ (see the discussion in Section \ref{sec:ev2}), by the direct relations between $E_{G}^P(p^m,\xi)$, $E_{F}^P(p^m,\xi)$ and $E_{f}^P(p^m,\xi)$, and by
$$
\hat \alpha_{F} = 1/2 +  \hat \alpha_{G}\ge \hat \alpha_{f} -\varepsilon/3\ge n/d -\varepsilon/3 = 4/3-\varepsilon/3,
$$
which follows from  (\ref{bound:min:exp}) and Example 6.7 of \cite{MustPopa}.

 If $G$ has multiplicity $2$ at $P$, then we may repeat the above reduction and reduce to the case that $G(y,z,w)=y^2+H(z,w)$ for some polynomial $H$ in two variables. Since now automatically $\hat \alpha_{H}\le 1$, we are similarly done by \cite[Theorem 1.5]{CMN}. This proves (\ref{m>1, p>M})  for all large primes $p$.

Small primes are treated similarly by Weierstrass preparation over $\ZZ_p$ rather than over $R$ and by a similar reasoning as for Lemma \ref{lem:T}. This finishes the proof of the theorem.
\end{proof}

In fact, the above proof of Theorem \ref{allp} gives a bit more in the case
$n\leq 4$, $d\le 3$, and $s=0$, as follows.

 \begin{prop}\label{prop:3}
If $n\leq 4$, $d\le 3$, and $s=0$, then the bounds from (\ref{bound:SfNfull}) hold for $f$, that is, for all $\varepsilon>0$ there is a constant $c$ such that
\begin{equation}\label{bound:3}
E_f(p^m,\xi) \le c (p^m)^{-\hat\alpha_f+\varepsilon} \mbox{ for all primes $p$, all $m>1$ and all $\xi$},
\end{equation}
with $\hat{\alpha}_f$ as in Section \ref{subs:minimal:exp}. Moreover, still under the conditions $n\leq 4$, $d\le 3$, and $s=0$, the value $\hat{\alpha}_f$ is equal to the motivic oscillation index of $f$ as given in \cite{CMN}.
Hence, $\hat{\alpha}_f$  is the optimal exponent in (\ref{bound:3}). 
\end{prop}
The optimality of the exponent  $\hat{\alpha}_f$ in (\ref{bound:3}) means that there is a constant $c_0>0$ such that for infinitely many primes $p$ and arbitrarily large $m$ one has
\begin{equation}\label{lowerb}
c_0(p^m)^{-\hat{\alpha}_f} \le E_f(p^m,\xi) \mbox{ for some $\xi$.}
\end{equation}
The motivic oscillation index of $f$ as given in \cite{CMN} (which corresponds to the one from \cite{Cigumodp} but without the negative sign) is the optimal exponent of $p^{-m}$ in (\ref{bound:3}), see Section 3.4 of \cite{CMN}; therefore, the equality of  $\hat{\alpha}_f$ with the motivic oscillation index is a useful property and implies (\ref{lowerb}).

\begin{proof}[Proof of Proposition \ref{prop:3}]
Under the conditions $n\leq 4$, $d\le 3$, and $s=0$, the proof of Theorem \ref{allp}  in fact shows the strengthening of (\ref{m>1, p>M}) with $\hat\alpha_f$ in the exponent instead of $n/d$. In each case treated in that proof, we observe that $\hat\alpha_f$ is equal to the motivic oscillation index of $f$. To this end we use that for any polynomial $G$ over $\ZZ$, the inequality $\hat \alpha_{G}\le 1$ implies that $\hat \alpha_{G}$ equals the motivic oscillation index of $G$ by \cite[Section 3.4]{CMN}. In the case that $f=f_3$, one uses that  $\hat\alpha_f=4/3$ (by \cite[Theorem~E (3)]{MustPopa}), and, that $4/3$ is optimal in the exponent in (\ref{m>1, p>M}) (see Section \ref{sec:ev1}), and hence, equal to the motivic oscillation index of $f$.
 \end{proof}

%

\begin{rem}\label{rem:largeprimes-n-d}
For each of the above cases in which Conjecture \ref{con1} is shown in this paper, one moreover sees that, after excluding a finite set $S$ (which depends on $f$) of prime divisors of $N$, the implied constant can be taken depending only on $d$ and $n$  (and on $\varepsilon$). The only case where this is not directly clear is for the case with $(n-s)/d\le 1$, since its treatment in \cite{CMN} uses a chosen log resolution which depends on $f$. However, the complexity  of such log resolutions (and of the corresponding proof in \cite{CMN}) remains bounded when $n$ and $d$ are fixed. Indeed, one first takes a log resolution of a generic polynomial of degree $d$ in $n$ variables; this then yields a log resolution for polynomials whose coefficients lie in a dense Zariski open subset $U$ of the parameter space. One proceeds similarly for a generic polynomial with parameters in the complement of the dense open $U$.

Note that the exclusion of a finite list of prime divisors of $N$ is necessary, as can be seen when one replaces a polynomial $f$ by $pf$ for some prime $p$. It is not clear at the moment whether the finite set $S$ has to depend fully on $f$ in general, or, just on $f_d$.
\end{rem}

 \bibliographystyle{amsplain}
\bibliography{anbib}

%
%
%

\end{document}